\documentclass{amsart}

\newtheorem{theorem}{Theorem}[section]
\newtheorem{lemma}[theorem]{Lemma}
\newtheorem{proposition}[theorem]{Proposition}
\newtheorem{corollary}[theorem]{Corollary}

\newtheorem{problem}[theorem]{Problem}

\theoremstyle{definition}

\newtheorem{remark}[theorem]{Remark}

\def\cocoa{{\hbox{\rm C\kern-.13em o\kern-.07em C\kern-.13em o\kern-.15em A}}}

\usepackage{amscd,amssymb}
\usepackage{youngtab}

\newenvironment{proofof}[1]{\noindent{\it Proof of
#1.}}{\hfill$\square$\\\mbox{}}

 \newenvironment{dedication}
        {\vspace{6ex}\begin{quotation}\begin{center}\begin{em}}
        {\par\end{em}\end{center}\end{quotation}}

\begin{document}

\title[Matrix valued concomitants of $\mathrm{SL}_2(\mathbb{C})$]
{Matrix valued concomitants of $\mathrm{SL}_2(\mathbb{C})$}

\author[M\'aty\'as Domokos]
{M\'aty\'as Domokos}
\address{Alfr\'ed R\'enyi Institute of Mathematics,
Re\'altanoda utca 13-15, 1053 Budapest, Hungary}
\email{domokos.matyas@renyi.hu}

\thanks{Partially supported by the Hungarian National Research, Development and Innovation Office,  NKFIH K 138828,  K 132002.}



\maketitle

\begin{dedication}
{Dedicated to Rich\'ard Wiegandt on his $90$th birthday}
\end{dedication}

\begin{abstract} 
To a finite dimensional representation of a complex Lie group $G$, an associative algebra of adjoint 
covariant polynomial maps from the direct sum of $m$ copies of the Lie algebra $\mathfrak{g}$ of $G$ into an algebra of complex matrices is associated. When the tangent representation of the given representation is irreducible, the center of this algebra of concomitants can be identified with the algebra of adjoint invariant polynomial functions on $m$-tuples of elements of $\mathfrak{g}$. For irreducible finite dimensional representations  of $\mathrm{SL}_2(\mathbb{C})$ minimal generating systems of the corresponding algebras of concomitants are determined, both as an algebra and as a module over its center. 
\end{abstract}

\section{Introduction} \label{sec:intro}

In this paper we shall study a problem that fits into the following general framework. Given a finite dimensional 
$\mathbb{C}$-vector space $V$ endowed with a linear action of a group $G$, and given an associative $\mathbb{C}$-algebra $A$, on which $G$ acts via $\mathbb{C}$-algebra automorphisms, denote by $\mathcal{C}_G(V,A)$ the set of $G$-equivariant polynomial maps from $V$ to $A$. A map $\varphi:V\to A$ is \emph{polynomial} if its image spans a finite dimensional subspace $W$ of $A$, and the coordinate functions of $\varphi:V\to W$ with respect to a basis of $W$ are polynomial functions on $V$. The set of all polynomial maps from $V$ to $A$ is naturally an associative  
$\mathbb{C}$-algebra with pointwise addition and multiplication of functions; that is, for polynomial maps $F_1,F_2:V\to A$,  $v\in V$, $c_1,c_2\in \mathbb{C}$ we set $(c_1F_1+c_2F_2)(v)=c_1(F_1(v))+c_2(F_2(v))$  and 
$(F_1F_2)(v)=F_1(v)F_2(v)$. Since the coordinate functions of $c_1F_1+c_2F_2$ are $\mathbb{C}$-linear combinations of the coordinate functions of $F_1$ and $F_2$, whereas the coordinate functions of $F_1F_2$ are $\mathbb{C}$-linear combinations of products $f_1f_2$, where $f_1$ is a coordinate function of $F_1$ and $f_2$ is a coordinate function of $F_2$,  the maps $c_1F_1+c_2F_2$ and $F_1F_2$ are polynomial maps. Moreover, the subset $\mathcal{C}_G(V,A)$ of the algebra of polynomial maps from $V$ to $A$ is a subalgebra, since $G$ acts on $A$ via $\mathbb{C}$-algebra automorphisms. 
We call the elements of $\mathcal{C}_G(V,A)$ the \emph{$A$-valued concomitants on $V$}. 

An instance of the above general setup is naturally associated to any finite dimensional representation 
$\Psi:G\to \mathrm{GL}(\mathbb{C}^d)$ of a complex Lie group $G$ with Lie algebra $\mathfrak{g}$. 
Namely, take 
\begin{itemize} 
\item $V:=\mathfrak{g}^{\oplus m}$, the direct sum of  $m$ copies of $\mathfrak{g}$, 
endowed with the $m$-fold direct sum $\mathrm{Ad}^{\oplus m}$ of the adjoint representation 
$\mathrm{Ad}:G\to \mathrm{GL}(\mathfrak{g})$ of $G$ on $\mathfrak{g}$; 
\item $A:=\mathrm{End}_{\mathbb{C}}(\mathbb{C}^d)=\mathbb{C}^{d\times d}$, the algebra of $d\times d$ 
complex matrices (identified with the algebra of $\mathbb{C}$-linear transformations of $\mathbb{C}^d$), on which 
$G$ acts via the representation 
$\widetilde\Psi:G\to\mathrm{GL}(\mathbb{C}^{d\times d})$ defined as follows: 
for $g\in G$ and $M\in \mathbb{C}^{d\times d}$ we have 
\begin{equation}\label{eq:widetildePsi}\widetilde\Psi(g)(M)=\Psi(g)M\Psi(g)^{-1}.
\end{equation}  
\end{itemize} 

\begin{problem}\label{problem:1} 
Describe the algebra $\mathcal{C}_G(\mathfrak{g}^{\oplus m},\mathbb{C}^{d\times d})$ 
(e.g. find its generators). 
\end{problem} 

In our work we shall deal with Problem~\ref{problem:1}  in the case below: 
\begin{itemize} 
\item We take as  $G$ the special linear group $\mathrm{SL}_2(\mathbb{C})$  of $2\times 2$ complex matrices having determinant $1$; 
then $V$ is $\mathfrak{sl}_2(\mathbb{C})^{\oplus m}$, the direct sum of  $m$ copies of the Lie algebra 
$\mathfrak{sl}_2(\mathbb{C})$ of $\mathrm{SL}_2(\mathbb{C})$ (consisting of $2\times 2$ traceless complex matrices) endowed with the $m$-fold direct sum 
$\mathrm{Ad}^{\oplus m}$ of the \emph{adjoint representation} 
\[\mathrm{Ad}:\mathrm{SL}_2(\mathbb{C})\to \mathrm{GL}(\mathfrak{sl}_2(\mathbb{C})),\quad  
\mathrm{Ad}(g)(X)=gXg^{-1}\]  
 (matrix multiplication) for $g\in \mathrm{SL}_2(\mathbb{C})$, $X\in \mathfrak{sl}_2(\mathbb{C})$; 
\item We take for $\Psi$ the $d$-dimensional irreducible representation 
\[\Psi^{(d)}:\mathrm{SL}_2(\mathbb{C})\to \mathrm{GL}(\mathbb{C}^d).\]    
So 
$A=\mathrm{End}_{\mathbb{C}}(\mathbb{C}^d)=\mathbb{C}^{d\times d}$, the algebra of $d\times d$ complex matrices, on which 
$\mathrm{SL}_2(\mathbb{C})$ acts via the representation 
$\widetilde\Psi^{(d)}:\mathrm{SL}_2(\mathbb{C})\to\mathrm{GL}(\mathbb{C}^{d\times d})$ defined by 
\eqref{eq:widetildePsi}. 
\end{itemize}
 
 Throughout this paper set 
 \begin{equation}\label{eq:C^{(d)}}\mathcal{C}^{(d)}:=
\mathcal{C}_{\mathrm{SL}_2(\mathbb{C})}(\mathfrak{sl}_2(\mathbb{C})^{\oplus m},\mathbb{C}^{d\times d}).\end{equation} 
It is known that for $d\le 3$ the "obvious" concomitants generate the $\mathbb{C}$-algebra 
$\mathcal{C}^{(d)}$, see Remark~\ref{remark:d<=3}. In our Theorem~\ref{thm:algebra generators} we shall extend this result   
for all $d$. 
In Theorem~\ref{thm:center} we show for general $G$ that if the tangent representation of $\Psi$ is irreducible, 
then the center of $\mathcal{C}_G(\mathfrak{g}^{\oplus  m},\mathbb{C}^{d\times d})$ is isomorphic to the algebra of 
$\mathrm{Ad}^{\oplus m}$-invariant polynomial functions on $\mathfrak{g}^{\oplus m}$. In particular, 
the center of the algebra $\mathcal{C}^{(d)}$ can be identified with $\mathcal{D}_0$, the algebra of $\mathrm{SL}_2(\mathbb{C})$-invariants of $m$-tuples of binary quadratic forms. We shall give a minimal $\mathcal{D}_0$-module generating system of $\mathcal{C}^{(d)}$ as well in Theorem~\ref{thm:D_0-module generators}. 
These results are deduced from classical known facts on the invariant theory of binary forms, with the aid of the representation theory of the general linear group.

\section{Adjoint invariants and the center of the concomitant algebra} 

In this section $G$ is any complex Lie group, $\mathfrak{g}$ is its Lie algebra, on which $G$ acts via its adjoint representation $\mathrm{Ad}:G\to \mathrm{GL}(\mathfrak{g})$. Let $\Psi:G\to \mathrm{GL}(\mathbb{C}^d)$ be a finite dimensional representation of $G$, and $\widetilde\Psi:G\to \mathrm{GL}(\mathbb{C}^{d\times d})$ the representation associated to $\Psi$ as in \eqref{eq:widetildePsi}. 
Denote by  $\rho:\mathfrak{g}\to \mathfrak{gl}(\mathbb{C}^d)=\mathbb{C}^{d\times d}$ the tangent representation of $\Psi$ (so $\rho$ is the differential of $\Psi$ at the identity element of $G$). We need the following basic fact: 

\begin{lemma} \label{lemma:Ad intertwiner} 
The linear map $\rho$  intertwines between the representation $\mathrm{Ad}$ of $G$ on $\mathfrak{g}$ and the representation $\widetilde\Psi$ of $G$ on $\mathbb{C}^{d\times d}$; that is,  
for $g\in G$ and $X\in \mathfrak{g}$ we have 
\[\rho(\mathrm{Ad}(g)(X))=\Psi(g) \rho(X)\Psi(g)^{-1} \text{ (matrix multiplication on the right hand side)}.\] 
\end{lemma} 
\begin{proof} 
Take $\gamma:\mathbb{R}\to G$ with $\frac{d}{dt}\gamma\vert_{t=0}=X$. 
By definition of the tangent representation we have 
\begin{align*}\rho(\mathrm{Ad}(g)(X))=\frac{d}{dt}\Psi(g\gamma(t)g^{-1})\vert_{t=0}=\frac{d}{dt}(\Psi(g)\Psi(\gamma(t))\Psi(g)^{-1})\vert_{t=0}
\\ =\Psi(g)\left(\frac{d}{dt}\Psi(\gamma(t)\vert_{t=0}\right) \Psi(g)^{-1}=\Psi(g)\rho(X)\Psi(g)^{-1}.\end{align*} 
\end{proof} 

\begin{remark} 
Note that 
Lemma~\ref{lemma:Ad intertwiner} implies that $\rho$ intertwines between the Lie algebra representations  
\[\mathrm{ad}:\mathfrak{g}\to \mathfrak{gl}(\mathfrak{g}), \quad 
\mathrm{ad}(X)(Y)=[X,Y] \text{ for }X,Y\in \mathfrak{g}\]  
and the representation 
\[\widetilde\rho:\mathfrak{g}\to \mathbb{C}^{d\times d},  \quad 
\widetilde\rho(X)(M)= \rho(X)M-M\rho(X) \]
for $X\in \mathfrak{g}$, $M\in \mathbb{C}^{d\times d}$.
This way we recover the obvious equality $\rho([X,Y])=\rho(X)\rho(Y)-\rho(Y)\rho(X)$ for $X,Y\in \mathfrak{g}$ 
(expressing that $\rho$ is a representation of the Lie algebra $\mathfrak{g}$). 
\end{remark}

Denote by 
\[R:=\mathbb{C}[\mathfrak{g}^{\oplus m}]^G\] 
the algebra of $\mathrm{Ad}^{\oplus m}$-invariant polynomial functions on $\mathfrak{g}^{\oplus m}$. 
Note that the map $f\mapsto fI_d$ (where $I_d$ is the $d\times d$ identity matrix) is an embedding of $R$ into the center of $\mathcal{C}_G(\mathfrak{g}^{\oplus m},\mathbb{C}^{d \times d})$.

\begin{proposition}\label{prop:noetherian} 
\begin{itemize} 
\item[(i)] The algebra $\mathcal{C}_G(\mathfrak{g}^{\oplus m},\mathbb{C}^{d \times d})$ is integral over its central subalgebra $RI_d$. 
\item[(ii)] If $G$ is reductive or $G$ is a maximal unipotent subgroup of a reductive group, then  
$\mathcal{C}_G(\mathfrak{g}^{\oplus m},\mathbb{C}^{d \times d})$ is a finitely generated module over its central subalgebra $RI_d$. 
In particular, the algebra $\mathcal{C}_G(\mathfrak{g}^{\oplus m},\mathbb{C}^{d \times d})$ is noetherian for such groups. 
\end{itemize}
\end{proposition} 

\begin{proof} (i) 
$\mathcal{C}_G(\mathfrak{g}^{\oplus m},\mathbb{C}^{d \times d})$ is naturally a subalgebra of the algebra 
$\mathbb{C}[\mathfrak{g}^{\oplus m}]^{d\times d}$ of $d\times d$ matrices over the commutative ring 
$\mathbb{C}[\mathfrak{g}^{\oplus m}]$. Now apply the Cayley-Hamilton Theorem to $f\in \mathbb{C}[\mathfrak{g}^{\oplus m}]$: the coefficients of the characteristic polynomial of $f$ are $G$-invariant polynomial functions on $\mathfrak{g}^{\oplus m}$, hence the Cayley-Hamilton identity  gives (i). 

(ii) It is well known that the assumptions on $G$ guarantee that 
for any finite dimensional $G$-modules $V$ and $W$, the space of $G$-equivariant polynomial maps from 
$V$ to $W$ is a finitely generated module over the finitely generated $\mathbb{C}$-algebra $\mathbb{C}[V]^G$ of $G$-invariant polynomial functions on $V$. Apply this for $V=\mathfrak{g}^{\oplus m}$ and $W=\mathbb{C}^{d\times d}$, 
and note that the embedding $R\to RI_d$ of $R$ into $\mathcal{C}_G(\mathfrak{g}^{\oplus m},\mathbb{C}^{d \times d})$ as a central subalgebra gives the $R$-module structure on $\mathcal{C}_G(\mathfrak{g}^{\oplus m},\mathbb{C}^{d \times d})$ mentioned above. Moreover, being a finitely generated module over a finitely generated commutative subalgebra, 
$\mathcal{C}_G(\mathfrak{g}^{\oplus m},\mathbb{C}^{d \times d})$ is noetherian. 
\end{proof}

\begin{theorem}\label{thm:center} 
If the tangent representation $\rho$ of $\Psi$ is irreducible, then the center of the concomitant algebra 
$\mathcal{C}_G(\mathfrak{g}^{\oplus m},\mathbb{C}^{d \times d})$ equals  $RI_d$ for $m\ge 2$. 
\end{theorem} 
\begin{proof} We generalize the proof of \cite[Proposition 4.4]{DD}  to this case. 
Since $\rho$ is irreducible, $\mathbb{C}^d$ is a faithful irreducible $\rho(\mathfrak{g})$-module. 
It follows that the Lie algebra $\rho(\mathfrak{g})$ is reductive, and $\mathbb{C}^d$ is an irreducible module over the semisimple Lie algebra $\rho([\mathfrak{g},\mathfrak{g}])$ (see for example the proof of \cite[Lemma 1.2]{nss}). 
By \cite[Theorem 6]{kuranishi} there are elements $X,Y\in \mathfrak{g}$ such that the Lie algebra 
$\rho([\mathfrak{g},\mathfrak{g}])$ is generated by $\rho(X)$ and $\rho(Y)$. Taking into account Burnside's Theorem on subalgebras over $\mathbb{C}^{d\times d}$ having no non-trivial invariant subspace in $\mathbb{C}^d$, 
we conclude that $\rho(X)$, $\rho(Y)$ generate the algebra $\mathbb{C}^{d\times d}$. 

For $i=1,\dots,m$ denote by $t_i:\mathfrak{g}^{\oplus m}\to \mathbb{C}^{d\times d}$ the map $(X_1,\dots,X_m)\mapsto \rho(X_i)$, where $\rho$ is the tangent representation of $\Psi$. Then 
$t_i\in \mathcal{C}_G(\mathfrak{g}^{\oplus m},\mathbb{C}^{d \times d})$ by Lemma~\ref{lemma:Ad intertwiner}.   
By the above paragraph there exists a finite set $J\subseteq \mathcal{C}_G(\mathfrak{g}^{\oplus m},\mathbb{C}^{d \times d})$ of products $t_{j_1}\cdots t_{j_s}$ with factors in $\{t_1,t_2\}$ such that 
$\{f(X,Y,0,\dots,0)\mid f\in J\}$ is a $\mathbb{C}$-vector space basis of $\mathbb{C}^{d\times d}$. Denoting by $\mathbb{F}$ the field of fractions of $\mathbb{C}[\mathfrak{g}^{\oplus m}]$, we conclude that $J$ is an $\mathbb{F}$-vector spaces basis of $\mathbb{F}^{d\times d}$. 
Indeed, consider the $d^2\times d^2$ matrix $M$ whose columns are the 
elements $f\in J\subset \mathbb{F}^{d\times d}$ displayed as a column vector in $\mathbb{F}^{d^2}$. 
In fact, $M\in \mathbb{C}[\mathfrak{g}^{\oplus m}]^{d^2\times d^2}$. 
If $J$ is linearly dependent over $\mathbb{F}$, then the determinant of  $M$ is zero, hence the complex matrix 
$M(X,Y,0,\dots,0)\in \mathbb{C}^{d^2\times d^2}$ (obtained from $M$ by the specialization $t_1\mapsto X$, $t_2\mapsto Y$) has determinant zero. This contradicts the assumption that   $\{f(X,Y,0,\dots,0)\mid f\in J\}$ is a $\mathbb{C}$-linearly independent subset of $\mathbb{C}^{d\times d}$. 

Now take a concomitant $c$ from the center of $\mathcal{C}_G(\mathfrak{g}^{\oplus m},\mathbb{C}^{d \times d})$. 
Then $c$ commutes with $t_1$ and $t_2$, hence $c$ commutes with each element from 
$J$. Since $J$ spans $\mathbb{F}^{d\times d}$, we conclude that $c$ is central in $\mathbb{F}^{d\times d}$. 
That is, $c(X)$ is a scalar matrix for any $X\in \mathfrak{g}^{\oplus m}$. The scalar belongs to the coordinate ring 
$\mathbb{C}[\mathfrak{g}^{\oplus m}]$, and taking into account that $c$ is a $G$-equivariant map, we conclude that 
$c=\bar cI_d$ for some $G$-invariant $\bar c\in R$. 
\end{proof}

\section{The main results on $\mathrm{SL}_2(\mathbb{C})$}  \label{sec:main results} 

Let us introduce some elements of $\mathcal{C}^{(d)}$ (defined in \eqref{eq:C^{(d)}}). 
Denote by $\rho^{(d)}:\mathfrak{sl}_2(\mathbb{C})\to \mathfrak{gl}(\mathbb{C}^d)=\mathbb{C}^{d\times d}$ the tangent representation of $\Psi^{(d)}$ (note that up to isomorphism, this is the unique irreducible $d$-dimensional representation of the Lie algebra $\mathfrak{sl}_2(\mathbb{C})$). 
For $i=1,\dots,m$ set 
\[t_i^{(d)}:\mathfrak{sl}_2(\mathbb{C})^{\oplus m}\to \mathbb{C}^{d\times d},\quad (X_1,\dots,X_m)\mapsto \rho^{(d)}(X_i).\] 
For $1\le i,j\le m$ set 
\[u_{ij}^{(d)}:\mathfrak{sl}_2(\mathbb{C})^{\oplus m}\to\mathbb{C}, \quad  (X_1,\dots,X_m)\mapsto \mathrm{Tr}(\rho^{(d)}(X_i)\rho^{(d)}(X_j))\] 
where $\mathrm{Tr}(-)$ stands for the trace function on matrices. 
For $1\le i,j,k\le m$ set 
\[w_{ijk}^{(d)}:\mathfrak{sl}_2(\mathbb{C})^{\oplus m}\to\mathbb{C}, \quad  (X_1,\dots,X_m)\mapsto \mathrm{Tr}(\rho^{(d)}(X_i)\rho^{(d)}(X_j)\rho^{(d)}(X_k)).\] 

By Lemma~\ref{lemma:Ad intertwiner} 
the elements $t_i^{(d)}$ belong to $\mathcal{C}^{(d)}$, and conjugation invariance of the trace function implies that  $u_{ij}^{(d)}$, $w_{ijk}^{(d)}$ belong to the algebra 
\begin{equation}\label{eq:D_0}\mathcal{D}_0:=\mathbb{C}[\mathfrak{sl}_2(\mathbb{C})^{\oplus m}]^{\mathrm{SL}_2(\mathbb{C})}\end{equation} 
of $\mathrm{Ad}^{\oplus m}$-invariants on  $\mathfrak{sl}_2(\mathbb{C})^{\oplus m}$. 

\begin{proposition}\label{prop:D_0generators}
For any $d\ge 2$, the elements $u_{ij}^{(d)}$ $(1\le i\le j\le m)$ and $w_{ijk}^{(d)}$ $(1\le i<j<k\le m)$ constitute a minimal $\mathbb{C}$-algebra generating system of $\mathcal{D}_0$. 
\end{proposition}  
\begin{proof} 
For the case $d=2$ see for example \cite{Procesi:2}, \cite{Le Bruyn}. 
The case $d=3$ is deduced in \cite{DD} from results in \cite{Procesi:1}. 
In fact the case $d=2$ shows that up to non-zero scalar multiples, for $1\le i\le j\le m$ there is a unique element of $\mathcal{D}_0$ bilinear in the $i$th and $j$th direct summands of $\mathfrak{sl}_2(\mathbb{C})^{\oplus m}$. Moreover, 
up to non-zero scalar multiples, for $1\le i<j<k \le m$ there is a unique element of $\mathcal{D}_0$ trilinear in the $i$th, $j$th, $k$th direct summands of $\mathfrak{sl}_2(\mathbb{C})^{\oplus m}$. Furthermore, these bilinear and trilinear invariants minimally generate $\mathcal{D}_0$. Therefore to prove the statement for general $d$, it is sufficient to show that the functions $u_{ij}^{(d)}$, $w_{ijk}^{(d)}$ are non-zero. This can be checked for example by direct computation: 
setting $h:=\left(\begin{array}{cc}1 & 0 \\0 & -1\end{array}\right)\in \mathfrak{sl}_2(\mathbb{C})$, we have that 
$\rho^{(d)}(h)$ is a non-zero diagonal matrix with integer entries (see the proof of 
Lemma~\ref{lemma:t_1^s}), hence $\mathrm{Tr}(\rho^{(d)}(h)^2)$ is a positive real number.  
This shows that $u_{ij}^{(d)}$ is not the zero function. The non-vanishing of $w_{ijk}^{(d)}$ follows from the proof 
of Lemma~\ref{lemma:w_{ijk}}.  
\end{proof} 

By Theorem~\ref{thm:center}, $\mathcal{D}_0I_d$ is the center of $\mathcal{C}^{(d)}$, hence we get the following corollary of Proposition~\ref{prop:D_0generators}:  

\begin{corollary}\label{cor:center-generators} 
The center of the algebra  $\mathcal{C}^{(d)}$ is minimally generated by the elements $u_{ij}^{(d)}I_d$ $(1\le i\le j\le m)$, $w_{ijk}^{(d)}I_d$ $(1\le i<j<k\le m)$.  
\end{corollary} 

\begin{theorem}\label{thm:algebra generators} 
For $d\ge 3$ the algebra $\mathcal{C}^{(d)}$ is minimally generated by the $\frac{m(m+3 )}{2}$ elements 
$t_1^{(d)},\dots,t_m^{(d)}$, $u_{ij}^{(d)}I_d$ $(1\le i\le j\le m)$. 
\end{theorem} 

\begin{remark} \label{remark:d<=3} 
For $d=1$ obviously $\mathcal{C}^{(1)}=\mathcal{D}_0$. 
The case $d=2$ is settled in \cite{Procesi:2}, where it is shown that $\mathcal{C}^{(2)}$ (which is no longer commutative) is minimally generated by $t_1^{(2)},\dots,t_m^{(2)}$. 
The case $d=3$ is discussed in \cite{DD} (although the result is not stated there in the above form). 
\end{remark} 

Identifying $\mathcal{D}_0$ with the center $\mathcal{D}_0I_d$ of $\mathcal{C}^{(d)}$, we can view $\mathcal{C}^{(d)}$ as a 
$\mathcal{D}_0$-module.  Proposition~\ref{prop:D_0generators} and Theorem~\ref{thm:algebra generators} 
imply that the products $t_{i_1}(d)\cdots t_{i_k}(d)$ generate $\mathcal{C}^{(d)}$ as a $\mathcal{D}_0$-module. 
This statement will be refined as follows. The general linear group $\mathrm{GL}_m(\mathbb{C})$ acts from the right on 
$\mathfrak{sl}_2(\mathbb{C})^{\oplus m}$: for $X=(X_1,\dots,X_m)\in \mathfrak{sl}_2(\mathbb{C})^{\oplus m}$ and 
$g=(g_{ij})_{i,j=1}^m$ we have 
\begin{equation}\label{eq:right GL-action}
X\cdot g=(\sum_{i=1}^mg_{i1}X_i,\dots,\sum_{i=1}^mg_{im}X_i).\end{equation}
This action commutes with the action of $\mathrm{SL}_2(\mathbb{C})$ on $\mathfrak{sl}_2(\mathbb{C})^{\oplus m}$, 
therefore we get an induced $\mathrm{GL}_m(\mathbb{C})$-action on $\mathcal{C}^{(d)}$ via $\mathbb{C}$-algebra automorphisms: for $F\in \mathcal{C}^{(d)}$, $X\in  \mathfrak{sl}_2(\mathbb{C})^{\oplus m}$, and $g\in \mathrm{GL}_m(\mathbb{C})$ we have 
\[gF: \mathfrak{sl}_2(\mathbb{C})^{\oplus m}\to \mathbb{C}^{d\times d},\quad X\mapsto F(X\cdot g).\] 
The standard grading on the algebra of polynomial functions on $ \mathfrak{sl}_2(\mathbb{C})^{\oplus m}$ 
induces a grading on $\mathcal{C}^{(d)}$ and on $\mathcal{D}_0$. The elements $t_i(d)$ are homogeneous of degree $1$, 
the elements $u_{ij}^{(d)}$ are homogeneous of degree $2$, and the elements $w_{ijk}^{(d)}$ are homogeneous of degree $3$. Denote by $\mathcal{D}_0^+$ the maximal ideal of $\mathcal{D}$ spanned by its homogeneous elements of positive degree. Then a set of homogeneous elements of $\mathcal{C}^{(d)}$ minimally generates the $\mathcal{D}_0$-module $\mathcal{C}^{(d)}$ if and only if they constitute a $\mathbb{C}$-vector space basis in a direct complement of the subspace $\mathcal{D}_0^+\mathcal{C}^{(d)}$ in $\mathcal{C}^{(d)}$. 

Recall that the isomorphism classes of the polynomial $\mathrm{GL}_m(\mathbb{C})$-modules are labeled by their \emph{highest weights}, which are identified with partitions $\lambda=(\lambda_1,\dots,\lambda_n)$, where 
$\lambda_1\ge\cdots\ge \lambda_m$ are nonnegative integers. 
For an element $F$ in a polynomial $\mathrm{GL}_m(\mathbb{C})$-module (e.g. for $F\in \mathcal{C}^{(d)}$) 
we write $\langle F\rangle_{\mathrm{GL}_m(\mathbb{C})}$ for the $\mathrm{GL}_m(\mathbb{C})$-module generated by $F$. We shall use the notation $[t_1^{(d)},t_2^{(d)}]:=t_1^{(d)}t_2^{(d)}-t_2^{(d)}t_1^{(d)}$.

\begin{theorem}\label{thm:D_0-module generators} 
Fix $d\ge 2$. 
In order to simplify notation, write $t_i:=t_i^{(d)}$. 
We have 
 \[\mathcal{C}^{(d)}=\mathcal{D}_0^+\mathcal{C}^{(d)}\oplus \bigoplus_{s=0}^{d-1}\langle t_1^s\rangle_{\mathrm{GL}_m(\mathbb{C})}\oplus\bigoplus_{s=1}^{d-1}\langle t_1^{s-1}[t_1,t_2]\rangle_{\mathrm{GL}_m(\mathbb{C})}.\] 
 In particular, $\mathcal{C}^{(d)}$ is minimally generated as a $\mathcal{D}_0$-module by the subspace 
 \[\bigoplus_{s=0}^{d-1}\langle t_1^s\rangle_{\mathrm{GL}_m(\mathbb{C})}\oplus\bigoplus_{s=1}^{d-1}\langle t_1^{s-1}[t_1,t_2]\rangle_{\mathrm{GL}_m(\mathbb{C})},\]
 where $\langle t_1^s\rangle_{\mathrm{GL}_m(\mathbb{C})}$ is an irreducible $\mathrm{GL}_m(\mathbb{C})$-submodule of $\mathcal{C}^{(d)}$ with highest weight $(s,0,\dots,0)$, 
and $\langle t_1^{s-1}[t_1,t_2]\rangle_{\mathrm{GL}_m(\mathbb{C})}$ is an irreducible $\mathrm{GL}_m(\mathbb{C})$-submodule of $\mathcal{C}^{(d)}$ with highest weight $(s,1,0,\dots,0)$.   
 \end{theorem} 

\begin{remark} (i) Theorem~\ref{thm:D_0-module generators}  in the special case $d=2$ follows from \cite{Procesi:2}, and in the special case $d=3$ it follows from \cite[Theorem 5.6]{DD}. 

(ii) The parallel preprint \cite{D} gives some information on the $\mathrm{GL}_m(\mathbb{C})$-module structure of 
$\mathcal{C}^{(d)}$. Namely, it is shown there that the multiplicity of the irreducible $\mathrm{GL}_m(\mathbb{C})$-module with highest weight $\lambda=(\lambda_1,\dots,\lambda_m)$ as a summand in $\mathcal{C}^{(d)}$ is non-zero only if $\lambda=(\lambda_1,\lambda_2,\lambda_3)$ (i.e. $\lambda$ has at most $3$ non-zero parts), and no multiplicity is greater than $3^{d-2}$ for $d\ge 2$. 
\end{remark} 

\section{Classical invariant theory}\label{sec:classical} 

In order to describe $\mathcal{C}^{(d)}$ we invoke known results on another algebra of concomitants, namely 
on 
\[\mathcal{D}=\mathcal{C}_{\mathrm{SL}_2(\mathbb{C})}(\mathfrak{sl}_2(\mathbb{C})^{\oplus m},
\mathbb{C}[x,y]),\] 
where $\mathrm{SL}_2(\mathbb{C})$ acts on the $2$-variable polynomial algebra $\mathbb{C}[x,y]$ as follows: 
for $g=\left(\begin{array}{cc}g_{11} & g_{12} \\g_{21} & g_{22}\end{array}\right)\in \mathrm{SL}_2(\mathbb{C})$ and 
$p\in\mathbb{C}[x,y]$ we set 
\[g\cdot p=p(g_{11}x+g_{21}y,g_{12}x+g_{22}y).\] 
For $q=0,1,2,\dots$ denote by $\mathbb{C}[x,y]_q$ the degree $q$ homogeneous component of $\mathbb{C}[x,y]$; 
this $(q+1)$-dimensional subspace is $\mathrm{SL}_2(\mathbb{C})$-invariant, on which $\mathrm{SL}_2(\mathbb{C})$ 
acts via the representation $\Psi^{(q+1)}$. The map 
\begin{equation}\label{eq:sl2=binary quadratic}\mathfrak{sl}_2(\mathbb{C})\to \mathbb{C}[x,y]_2,\quad 
\left(\begin{array}{cc}b & -a\\ c & -b\end{array}\right)\mapsto ax^2+2bxy+cy^2\end{equation}  
is an $\mathrm{SL}_2(\mathbb{C})$-module isomorphism. 
Using this isomorphism we make the identification 
\[\mathcal{D}=\mathcal{C}_{\mathrm{SL}_2(\mathbb{C})}(\mathbb{C}[x,y]_2^{\oplus m},\mathbb{C}[x,y]).\] 

\begin{remark}\label{remark:D} 
The algebra $\mathcal{D}$ appears in classical invariant theory under the name of \emph{the algebra of joint concomitants of several binary quadratic forms}. The commutativity of $\mathbb{C}[x,y]$ implies that $\mathcal{D}$ is commutative, unlike  the algebras $\mathcal{C}^{(d)}$. 
\end{remark} 
 
 There is a natural bigrading on $\mathcal{D}$; namely, denote by $\mathcal{D}_{q,p}$ the space of degree 
 $p$ homogeneous $\mathrm{SL}_2(\mathbb{C})$-equivariant polynomial maps from 
 $\mathcal{C}[x,y]_2^{\oplus m}$ to $\mathbb{C}[x,y]_q$. 
 Then we have 
 \[\mathcal{D}=\bigoplus_{q=0}^\infty \mathcal{D}_q,\ \text{ where }\ 
 \mathcal{D}_q=\bigoplus_{p=0}^{\infty} \mathcal{D}_{q,p}.\] 
 For a non-zero concomitant $F\in\mathcal{D}_{q,p}$, the number $q$ is called the \emph{order} of $F$, and $p$ is called 
 the \emph{degree} of $F$. The subspace $\mathcal{D}_0$ is a subalgebra of $\mathcal{D}$, and because of the identification \eqref{eq:sl2=binary quadratic}, it is identified with the algebra  
 $\mathcal{D}_0=\mathbb{C}[\mathfrak{sl}_2(\mathbb{C})^{\oplus m}]^{\mathrm{SL}_2(\mathbb{C})}$ 
introduced  in \eqref{eq:D_0} in Section~\ref{sec:main results}. 
Moreover, $\mathcal{D}_q$ is a $\mathcal{D}_0$-submodule of $\mathcal{D}$ for all $q$.   
 
 Let us introduce some elements of $\mathcal{D}$. 
 For $i=1,\dots,m$ set 
 \[F_i:\mathbb{C}[x,y]_2^{\oplus m}\to \mathbb{C}[x,y]_2,\quad (f_1,\dots,f_m)\mapsto f_i.\] 
 The \emph{discriminant} of a binary quadratic form $f=ax^2+2bxy+cy^2$ is 
 $\mathrm{Disc}(f)=4(b^2-ac)$.  For $1\le i,j\le m$ set 
 \[D_{ij}:\mathbb{C}[x,y]_2^{\oplus m}\to \mathbb{C},\quad (f_1,\dots,f_m)\mapsto 
 \frac 18(\mathrm{Disc}(f_i+f_j)-\mathrm{Disc}(f_i)-\mathrm{Disc}(f_j)).\] 
 In particular, $D_{ii}(f_1,\dots,f_m)=\frac 14\mathrm{Disc}(f_i)$. 
 For a triple of binary forms, $f_i=a_ix^2+2b_ixy+c_iy^2$, $i=1,2,3$, set 
 \[E(f_1,f_2,f_3):=\det\left(\begin{array}{ccc}a_1 & b_1 & c_1 \\a_2 & b_2 & c_2 \\a_3 & b_3 & c_3\end{array}\right).\] 
 For $1\le i,j,k\le m$ set 
 \[E_{ijk}:\mathbb{C}[x,y]_2^{\oplus m}\to \mathbb{C},\quad (f_1,\dots,f_m)\mapsto 
 E(f_i,f_j,f_k).\] 
 The \emph{Jacobian} of the binary quadratic forms $f_1$ and $f_2$ is the binary quadratic form 
 \[\frac 14\mathrm{Jac}(f_1,f_2)=
 (a_1b_2-b_1a_2)x^2+(a_1c_2-c_1a_2)xy+(b_1c_2-c_1b_2)y^2.\] 
 For $1\le i,j\le m$ set 
 \[J_{ij}:\mathbb{C}[x,y]_2^{\oplus m}\to \mathbb{C}[x,y]_2, \quad (f_1,\dots,f_m)\mapsto \frac 14\mathrm{Jac}(f_i,f_j).\] 
  
 In the following theorem we collect the known facts on the generators and relations of the algebra $\mathcal{D}$ that we shall need later. 
 
 \begin{theorem} \label{thm:grace and young} 
 \begin{itemize} 
 \item[(i)] The $\mathbb{C}$-algebra $\mathcal{D}$ is generated by the elements $F_i$ $(i=1,\dots,m)$, 
 $D_{ij}$ $(1\le i\le j\le m)$, $E_{ijk}$ $(1\le i<j<k\le m)$, and 
 $J_{ij}$ $(1\le i<j\le m)$. 
 \item[(ii)] We have the equality 
 \[J_{12}J_{34}=-D_{13}F_2F_4-D_{24}F_1F_3+D_{14}F_2F_3+D_{23}F_1F_4.\] 
\item[(iii)] We have the equality 
\[F_1J_{23}-F_2J_{13}+F_3J_{12}=0.\] 
\end{itemize}
\end{theorem}  
 \begin{proof} For (i) and (ii) see for example \cite[Page 162]{Grace-Young} 
 (a concomitant denoted by the same letter in loc. cit. as here in some cases differs by a non-zero scalar factor from ours). 
 It is claimed on page 164, Section 139  of loc. cit. that "every kind of syzygy which occurs in the irreducible system of concomitants for any number of quadratics has now been stated". The relation (iii) is not stated there, but one can easily verify that it holds as well. 
 \end{proof} 
 
 The right action of $\mathrm{GL}_m(\mathbb{C})$ on $\mathbb{C}[x,y]_2^{\oplus m}$ given by \eqref{eq:right GL-action} 
(recall that we have identified $\mathbb{C}[x,y]_2^{\oplus m}$ with $\mathbb{C}[\mathfrak{sl}_2(\mathbb{C})^{\oplus m}]$) induces an action on $\mathcal{D}$ via $\mathbb{C}$-algebra automorphisms: for 
 $F\in \mathcal{D}$, $X\in \mathbb{C}[x,y]_2^{\oplus m}$, and $g\in \mathrm{GL}_m(\mathbb{C})$ we have 
 $(gF)(X)=F(X\cdot g)$. 
 
 \begin{corollary}\label{cor:odd q} 
 \begin{itemize} 
 \item[(i)] For odd $q$ we have $\mathcal{D}_q=\{0\}$. 
 \item[(ii)]  For $q=2s>0$ even  we have  
 \[\mathcal{D}_{q,p}=\begin{cases} \{0\} &\text{ for }p<s; \\
 \mathrm{Span}_{\mathbb{C}}\{F_{i_1}\cdots F_{i_s}\mid 1\le i_1,\dots,i_s\le m\}&\text{ for }p=s;\\
 \mathrm{Span}_{\mathbb{C}}\{F_{i_1}\cdots F_{i_{s-1}}J_{kl}\mid 1\le i_1,\dots,i_{s-1}\le m, 1\le k<l\le m\}&\text{ for }p=s+1.
 \end{cases}\] 
\item[(iii)] For $q=2s>0$ even we have 
 \[\mathcal{D}_q=\mathcal{D}_0^+\mathcal{D}_q\oplus \mathcal{D}_{q,s}
 \oplus \mathcal{D}_{q,s+1}.\]  
 In particular, $\mathcal{D}_q$ is minimally generated as a $\mathcal{D}_0$-module by its subspace 
 $\mathcal{D}_{q,s}
 \oplus \mathcal{D}_{q,s+1}$. 
 \item[(iv)] For $q=2s$ even, $\mathcal{D}_{q,s}=\langle F_1^s\rangle_{\mathrm{GL}_m(\mathbb{C})}$ and 
 $\mathcal{D}_{q,s+1}=\langle F_1^{s-1}J_{12}\rangle_{\mathrm{GL}_m(\mathbb{C})}$ are irreducible 
 $\mathrm{GL}_m(\mathbb{C})$-submodules of $\mathcal{D}$ of highest weight $(s,0,\dots,0)$ and 
 $(s,1,0,\dots,0)$, respectively. 
 \end{itemize}
 \end{corollary} 
 
 \begin{proof}
 (i) The generators of $\mathcal{D}$ given in Theorem~\ref{thm:grace and young} have even order, hence 
 $\mathcal{D}_q=\{0\}$ for odd $q$. 
 
 (ii) and (iii): The bidegrees of the generators of the algebra $\mathcal{D}$ are the following:   
 \begin{equation*}\label{eq:bidegree}F_i\in \mathcal{D}_{2,1}, \quad D_{ij}\in \mathcal{D}_{0,2}, \quad E_{ijk}\in \mathcal{D}_{0,3}, 
 \quad J_{ij}\in \mathcal{D}_{2,2}.\end{equation*} 
Therefore  a product of the generators of $\mathcal{D}$ has order $q=2s$ if and only if it has $s$ factors of the type 
  $F_i$ or $J_{kl}$. Taking into account Theorem~\ref{thm:grace and young} (ii) we get that $J_{kl}J_{nr}$ can be rewritten as a $\mathcal{D}_0^+$-linear combination of products $F_iF_j$. Thus (ii) and (iii) obviously follow. 
  
  (iv) Denote by $U$ the subgroup of unipotent upper triangular matrices in $\mathrm{GL}_m(\mathbb{C})$, and denote by $T$ the subgroup of diagonal matrices in $\mathrm{GL}_m(\mathbb{C})$.  
  Then $F_1$ is obviously $U$-invariant. Moreover, $\mathrm{Jac}(f_1,f_2+af_1)=\mathrm{Jac}(f_1,f_2)$ 
  for any $a\in \mathbb{C}$. Thus $J_{12}$ is also $U$-invariant. For 
  $z=\mathrm{diag}(z_1,\dots,z_m)\in T$ we have 
  $z\cdot F_1^s=z_1^s F_1^s$ and $z\cdot (F_1^{s-1}J_{12})=z_1^sz_2F_1^{s-1}J_{12}$. 
  These calculations show that $F_1^s$ and $F_1^{s-1}J_{12}$ are highest weight vectors for $GL(\mathbb{C}^m)$ with weight $(s,0,\dots,0)$ and $(s,1,0,\dots,0)$, respectively. 
  
  The $\mathrm{GL}_m(\mathbb{C})$-action preserves the bigrading on $\mathcal{D}$. 
  Since $F_1^s\in \mathcal{D}_{q,s}$ and $F_1^{s-1}J_{12}\in \mathcal{D}_{q,s+1}$, we 
  conclude that $\langle F_1^s\rangle_{\mathrm{GL}_m(\mathbb{C})}\subseteq \mathcal{D}_{q,s}$ and 
  $\langle F_1^{s-1}J_{12}\rangle_{\mathrm{GL}_m(\mathbb{C})}\subseteq \mathcal{D}_{q,s+1}$. 
  By (iii) the space $\mathcal{D}_{q,s}$ is spanned by momomials of degree $s$ in $F_1,\dots,F_m$. The number of such monomials is $\binom{s+m-1}{s}$. This is equal to the dimension of the irreducible $\mathrm{GL}_m(\mathbb{C})$-module with highest weight $(s,0,\dots,0)$, implying the equality 
 $\langle F_1^s\rangle_{\mathrm{GL}_m(\mathbb{C})}= \mathcal{D}_{q,s}$. 
 The dimension of the irreducible $\mathrm{GL}_m(\mathbb{C})$-module with highest weight $(s,1,0,\dots,0)$ equals the number of semistandard tableaux of shape $(s,1)$ with content $\{1,\dots,m\}$; the set of such tableaux is 
 in a natural bijection with the set 
 $\mathcal{T}:=\{(i_1,\dots,i_s,j)\mid 1\le i_1\le\cdots\le i_s\le m,\ i_1<j\le m\}$. By (iii) the space $\mathcal{D}_{q,s+1}$ is 
spanned by elements of the form $F_{i_2}\cdots F_{i_s}J_{kl}$, where $1\le i_2\le \cdots \le i_s\le m$ and $k<l$. 
If $k>i_2$, then we may apply Theorem~\ref{thm:grace and young} (iii), asserting that 
$F_{i_2}J_{kl}=F_kJ_{i_2l}-F_lJ_{i_2,k}$. Therefore 
\[F_{i_2}\cdots F_{i_s}J_{kl}=F_kF_{i_3}\cdots J_{i_2l}-F_lF_{i_3}\cdots F_{i_s}J_{i_2k}.\]  
That is, $F_{i_2}\cdots F_{i_s}J_{kl}$ can be rewritten as a linear combination of elements of the form 
$F_{j_2}\cdots F_{j_s}J_{nr}$, where $(n,j_2,\dots,j_s,r)\in \mathcal{T}$. 
Consequently, 
$\dim_{\mathbb{C}}(\mathcal{D}_{q,s+1})\le |\mathcal{T}|=\dim_{\mathbb{C}}(\langle F_1^{s-1}J_{12}\rangle_{\mathrm{GL}_m(\mathbb{C})})$, 
implying in turn that 
$\mathcal{D}_{q,s+1}=\langle F_1^{s-1}J_{12}\rangle_{\mathrm{GL}_m(\mathbb{C})}$. 
 \end{proof} 
 
 \section{Proof of the main results on $\mathrm{SL}_2(\mathbb{C})$} \label{sec:main proof} 
 
 We have 
 \[\widetilde\Psi^{(d)}\cong \Psi^{(d)}\otimes(\Psi^{(d)})^*\cong \bigoplus_{q=0}^{d-1}\Psi^{(2q+1)}\] 
 by the Clebsch-Gordan rules. Consider the corresponding decomposition 
 \begin{equation*}\label{eq:clebsch-gordan}
 \mathbb{C}^{d\times d}=V_1\oplus V_3\oplus \cdots\oplus V_{2d-1}
 \end{equation*} 
 as a direct sum of minimal $\widetilde\Psi^{(d)}$-invariant subspaces, so 
 \begin{equation}\label{eq:V=C[x,y]} V_{2s+1}\cong \mathbb{C}[x,y]_{2s}\end{equation}
 as $\mathrm{SL}_2(\mathbb{C})$-modules. 
 For $s=0,1,\dots,d-1$ denote by $\mathcal{C}^{(d)}_{2s}$ the subspace of $\mathcal{C}^{(d)}$ consisting of the $\mathrm{SL}_2(\mathbb{C})$-equivariant polynomial maps from $\mathfrak{sl}_2(\mathbb{C})^{\oplus m}$ to 
 $V_{2s+1}$. Note that $\mathcal{C}^{(d)}_0$ coincides with $\mathcal{D}_0$ introduced in \eqref{eq:D_0} in Section~\ref{sec:main results}, $\mathcal{C}^{(d)}_{2s}$ is a graded $\mathcal{D}_0$-submodule of $\mathcal{C}^{(d)}$, where we endow $\mathcal{C}^{(d)}$ with the grading induced by the standard grading of the algebra of polynomial functions on $\mathfrak{sl}_2(\mathbb{C})^{\oplus m}$, 
 and  we have 
 \begin{equation}\label{eq:C^d-decomp}
 \mathcal{C}^{(d)}=\bigoplus_{s=0}^{d-1}\mathcal{C}^{(d)}_{2s},
 \end{equation} 
 a decomposition into the direct sum of graded $\mathcal{D}_0$-submodules.  
 Moreover, the identifications \eqref{eq:sl2=binary quadratic} and \eqref{eq:V=C[x,y]} induce the isomorphism 
 \begin{equation}\label{eq:C=D} 
\varphi: \mathcal{C}^{(d)}_{2s}\stackrel{\cong}\longrightarrow \mathcal{D}_{2s}
 \end{equation} 
 both as graded $\mathcal{D}_0$-modules and as $\mathrm{GL}_m(\mathbb{C})$-modules. 
 
 For the rest of this section we write $t_i:=t_i^{(d)}$, $i=1,\dots,m$. 
 \begin{lemma}\label{lemma:t_1^s} 
 Both $t_1^s$ and $t_1^{s-1}[t_1,t_2]$ are contained in 
 $\bigoplus_{j=0}^s\mathcal{C}^{(d)}_{2j}\setminus \bigoplus_{j=0}^{s-1}\mathcal{C}^{(d)}_{2j}$ for $s=1,\dots,d-1$. 
 \end{lemma} 
 
 \begin{proof} 
Take the basis 
\[e:=\left(\begin{array}{cc}0 & 1 \\0 & 0\end{array}\right), \quad f:=\left(\begin{array}{cc}0 & 0 \\1 & 0\end{array}\right), 
\quad h:=\left(\begin{array}{cc}1 & 0 \\0 & -1\end{array}\right)\]
of the Lie algebra $\mathfrak{sl}_2(\mathbb{C})$, so $[h,e]=2e$, $[h,f]=-2f$, and $[e,f]=h$.  
Identify the space $\mathbb{C}[x,y]_{d-1}$ of degree $(d-1)$ binary forms with $\mathbb{C}^d$ by choosing the basis 
$x^{d-1},x^{d-2}y,\dots,y^{d-1}$. 
Denote by $E_{i,j}$ the matrix unit with entry $1$ in the $(i,j)$ position and zeros in all other positions. 
The representation $\rho^{(d)}:\mathfrak{sl}_2(\mathbb{C})\to \mathfrak{gl}(\mathbb{C}[x,y]_{d-1})=\mathfrak{gl}(\mathbb{C}^d)$ takes the 
following matrix values on $e,f,h$: 
\[\rho^{(d)}(e)=\sum_{i=1}^{d-1}iE_{i,i+1},\quad \rho^{(d)}(f)=\sum_{i=1}^{d-1}(d-i)E_{i+1,1}, \quad 
\rho^{(d)}(h)=\sum_{i=1}^d(d+1-2i)E_{i,i}.\] 
This shows that for any $X\in \mathfrak{sl}_2(\mathbb{C})$, the matrix $\rho^{(d)}(X)$ is tridiagonal. 
Therefore for $k\in \{0,1,\dots,d-1\}$ and $X_1,\dots,X_k\in \mathfrak{sl}_2(\mathbb{C})$ the $(i,j)$-entry of 
$\rho^{(d)}(X_1)\cdots\rho^{(d)}(X_k)$ is $0$ when $|i-j|>k$. 

Denote by $H=\{\mathrm{diag}(z,z^{-1})\mid z\in \mathbb{C}^{\times}\}$ the diagonal subgroup of $\mathrm{SL}_2(\mathbb{C})$. 
We have 
$\Psi^{(d)}(\mathrm{diag}(z,z^{-1}))=\mathrm{diag}(z^{d-1},z^{d-3},\dots,z^{-d+1})$ and 
\[\widetilde\Psi^{(d)}(\mathrm{diag}(z,z^{-1}))(E_{i,j})=z^{2(j-i)}E_{i,j},\]
so $E_{i,j}$ is an $H$-weight vector with weight $2(j-i)$.  
It follows that the  $\widetilde\Psi^{(d)}$-invariant subspace 
\[M_s:=\mathrm{Span}_{\mathbb{C}}\{\rho^{(d)}(X_1)\cdots\rho^{(d)}(X_s)\mid 
X_1,\dots,X_s\in \mathfrak{sl}_2(\mathbb{C})\}\] 
is contained in the sum of $H$-weight subspaces of $\mathbb{C}^{d\times d}$ with weight 
$k$, where $|k|\le 2s$. In particular, the highest weight vectors in $V_{2s+3},\dots,V_{2d-1}$ do not belong to $M_s$. 
Since  $V_{2s+3},\dots,V_{2d-1}$ are minimal $\mathrm{SL}_2(\mathbb{C})$-invariant subspaces and they are pairwise non-isomorphis, we conclude that 
\[M_s\cap V_{2s+3}\oplus\cdots\oplus V_{2d-1}=\{0\},\]
implying (as $\mathbb{C}^{d\times d}$ is a multiplicity-free $\mathrm{SL}_2(\mathbb{C})$-module) that 
\[M_s\subseteq V_1\oplus V_3\oplus\cdots \oplus V_{2s+1}.\] 
The image of the concomitant $t_1^s:\mathfrak{sl}_2(\mathbb{C})^{\oplus m}\to \mathbb{C}^{d\times d}$ is 
contained in $M_s$, thus we conclude that $t_1^s\in \bigoplus_{j=0}^s\mathcal{C}^{(d)}_{2j}$. 
Similarly, since $[\rho^{(d)}(X_1),\rho^{(d)}(X_2)]=\rho^{(d)}([X_1,X_2])$, the image of the concomitant 
$t_1^{s-1}[t_1,t_2]$ is also contained in $M_s$, and so 
$t_1^{s-1}[t_1,t_2]\in \bigoplus_{j=0}^s\mathcal{C}^{(d)}_{2j}$ as well. 

Observe finally that $t_1^s(e,0,\dots,0)$ and $(t^{s-1}[t_1,t_2])(e,h,0,\dots,0)$  are non-zero scalar multiples of 
$\rho^{(d)}(e)^s$, hence they commute with $\rho^{(d)}(e)$; furthermore, they are $H$-weight vectors of weight $2s+1$. 
Therefore they are 
highest weight vectors in 
the $\mathrm{SL}_2(\mathbb{C})$-module $\mathbb{C}^{d\times d}$ with weight $2s+1$. 
It follows that none of  the images of the concomitants $t_1^s$ or $t^{s-1}[t_1,t_2]$ is contained in 
$\bigoplus_{j=0}^{s-1}V_{2j+1}$, 
implying in turn that none of $t_1^s$ or $t^{s-1}[t_1,t_2]$ is contained in 
$\bigoplus_{j=0}^{s-1}\mathcal{C}^{(d)}_{2j}$. 
 \end{proof} 
  
 \begin{proofof}{Theorem~\ref{thm:D_0-module generators}}
The elements $t_1^s$ and $t_1^{s-1}[t_1,t_2]$ are fixed by the subgroup of unipotent upper triangular matrices in $\mathrm{GL}_m(\mathbb{C})$, and they are eigenvectors of the diagonal subgroup of $\mathrm{GL}_m(\mathbb{C})$ 
 with weight $(s,0,\dots,0)$ and $(s,1,0,\dots,0)$, respectively.  So they are $\mathrm{GL}_m(\mathbb{C})$-highest weight vectors, and generate irreducible $\mathrm{GL}_m(\mathbb{C})$-submodules in $\mathcal{C}^{(d)}$. 
 
 By \eqref{eq:C=D}  and \eqref{eq:C^d-decomp} it is clearly sufficient to show that for $s=1,\dots,d-1$ we have 
 \begin{equation}\label{eq:C decomp} 
 \bigoplus_{j=0}^s\mathcal{C}^{(d)}_{2j}=\bigoplus_{j=0}^{s-1}\mathcal{C}^{(d)}_{2j}\oplus 
 \mathcal{D}_0^+\mathcal{C}^{(d)}_{2s}
 \oplus\langle t_1^s\rangle_{\mathrm{GL}_m(\mathbb{C})}\oplus\langle t_1^{s-1}[t_1,t_2]\rangle_{\mathrm{GL}_m(\mathbb{C})}. 
 \end{equation}
 Using the isomorphism $\varphi$ from \eqref{eq:C=D} we can transfer the problem to the $\mathcal{D}_0$-module 
 $\mathcal{D}$, described in Section~\ref{sec:classical}. Set $u:=\varphi(t_1^s)$ and $w:=\varphi(t_1^{s-1}[t_1,t_2])$. 
 By Lemma~\ref{lemma:t_1^s} we have 
 $u\in \bigoplus_{j=0}^s\mathcal{D}_{2j,s}\setminus \bigoplus_{j=0}^{s-1}\mathcal{D}_{2j}$ and 
 $w\in  \bigoplus_{j=0}^s\mathcal{D}_{2j,s+1}\setminus \bigoplus_{j=0}^{s-1}\mathcal{D}_{2j}$. 
 By Corollary~\ref{cor:odd q} (iii) the degree $s$ and degree $s+1$ homogeneous components of 
 $\mathcal{D}_0^+\mathcal{D}_{2s}$ are zero, therefore 
 \[u,w\notin \bigoplus_{j=0}^{s-1}\mathcal{D}_{2j}\oplus \mathcal{D}_0^+\mathcal{D}_{2s}.\]
 The space on the right hand side above is a $\mathrm{GL}_m(\mathbb{C})$-submodule in $\mathcal{D}$, 
and since $\varphi$ is a $\mathrm{GL}_m(\mathbb{C})$-module isomorphism, 
 $\langle u\rangle_{\mathrm{GL}_m(\mathbb{C})}$ and $\langle w\rangle_{\mathrm{GL}_m(\mathbb{C})}$ are non-isomorphic minimal non-zero $\mathrm{GL}_m(\mathbb{C})$-invariant subspaces. 
 Therefore 
 \begin{equation}\label{eq:cap} \langle u\rangle_{\mathrm{GL}_m(\mathbb{C})}\oplus \langle w\rangle_{\mathrm{GL}_m(\mathbb{C})}
 \cap \bigoplus_{j=0}^{s-1}\mathcal{D}_{2j}\oplus \mathcal{D}_0^+\mathcal{D}_{2s}=\{0\}.
 \end{equation} 
 On the other hand, by Corollary~\ref{cor:odd q} (iii) and (iv) we have 
 \begin{equation}\label{eq:factor} 
 \bigoplus_{j=0}^s\mathcal{D}_{2j}/(\bigoplus_{j=0}^{s-1}\mathcal{D}_{2j}\oplus \mathcal{D}_0^+\mathcal{D}_{2s})
 \cong \langle u\rangle_{\mathrm{GL}_m(\mathbb{C})}\oplus \langle w\rangle_{\mathrm{GL}_m(\mathbb{C})} 
\end{equation}
as $\mathbb{C}$-vector spaces (in fact as $\mathrm{GL}_m(\mathbb{C})$-modules). 
Combining \eqref{eq:cap} and \eqref{eq:factor} we conclude that 
\[\bigoplus_{j=0}^s\mathcal{D}_{2j}=\bigoplus_{j=0}^{s-1}\mathcal{D}_{2j}\oplus 
 \mathcal{D}_0^+\mathcal{D}_{2s}
 \oplus\langle u \rangle_{\mathrm{GL}_m(\mathbb{C})}\oplus\langle w \rangle_{\mathrm{GL}_m(\mathbb{C})}. 
\] 
The isomorphism $\varphi^{-1}$ (see \eqref{eq:C=D}) applied to the above equality gives the desired \eqref{eq:C decomp}, 
finishing the proof. 
 \end{proofof} 
 
 \begin{proofof}{Theorem~\ref{thm:algebra generators}}
 Theorem~\ref{thm:D_0-module generators} and Proposition~\ref{prop:D_0generators} imply that the algebra 
 $\mathcal{C}^{(d)}$ is generated by $t_1,\dots,t_m$, $u_{ij}^{(d)}I_d$ ($1\le i\le j\le m$), and 
 $w_{ijk}^{(d)}I_d$ ($1\le i<j<k\le m$).  The generators $w_{ijk}^{(d)}I_d$  are redundant 
by Lemma~\ref{lemma:w_{ijk}} below, so we can omit them. It remains to show that we are left with a minimal generating system. 
The generators $t_1,\dots,t_m$ have degree $1$, the minimal possible positive degree, and they are linearly independent. 
The other generators have degree $2$, and they span an irreducible $\mathrm{GL}_m(\mathbb{C})$-module with highest weigh $(2,0,...,0)$, the element $u_{11}^{(d)}I_d$ being the corresponding highest weight vector. 
Now the linear span of the quadratic products in $t_1,\dots,t_m$ contains only one (up to non-zero scalar multiples) highest weight vector of weight $(2,0,\dots,0)$, namely $t_1^2$, and for $d\ge 3$, $t_1^2$ does not belong to 
$\mathcal{D}_0I_d$. Therefore $\mathrm{Span}_{\mathbb{C}}\{u_{ij}^{(d)}I_d\mid 1\le i\le j\le m\}$ is disjoint from the subalgebra of $\mathcal{C}^{(d)}$ generated by $t_1,\dots,t_m$. This finishes the proof of the minimality of the generating system of $\mathcal{C}^{(d)}$ given in the theorem. 
 \end{proofof}

\begin{lemma}\label{lemma:w_{ijk}} 
We have the equality 
\[(d^2-1)w_{ijk}^{(d)}I_d=\sum_{i=1}^{\lfloor \frac{d+1}{2}\rfloor}(d+1-2i)^2\mathrm{St}_3(t_1,t_2,t_3)\in \mathcal{C}^{(d)},\]  
where 
\begin{align*}\mathrm{St}_3(t_1,t_2,t_3)=
 t_1t_2t_3+t_2t_3t_1+t_3t_1t_2  
-t_2t_1t_3-t_1t_3t_2-t_3t_2t_1.
\end{align*} 
\end{lemma} 

\begin{proof} 
We have the equality 
$t_4\mathrm{St}_3(t_1,t_2,t_3)=\mathrm{St}_3(t_1,t_2,t_3)t_4$ 
by \cite[Theorem 38.1]{razmyslov} (giving a basis of the polynomial identities satisfied by the irreducible representations of $\mathfrak{sl}_2(\mathbb{C})$). 
It follows that the image of $\mathrm{St}_3(t_1,t_2,t_3)$ in $\mathbb{C}^{d\times d}$ is contained in the centralizer of 
$\rho^{(d)}(\mathfrak{sl}_2(\mathbb{C}))$. Hence by Schur's Lemma we have $\mathrm{St}_3(t_1,t_2,t_3)=fI_d$ for some 
$f\in \mathbb{C}[\mathfrak{sl}_2(\mathbb{C})^{\oplus m}]$. Moreover, since $\mathrm{St}_3(t_1,t_2,t_3)$ is a concomitant, $f\in \mathcal{D}_0$. 
However, for each $1\le i<j<k\le m$, up to scalar multiples there is only one trilinear element in $\mathcal{D}_0$, 
that is linear in the $i$th, $j$th, and $k$th direct summands in 
$\mathfrak{sl}_2(\mathbb{C})^{\oplus m}$.  
Consequently, $\mathrm{St}_3(t_1,t_2,t_3)$ is a scalar multiple of $w_{ijk}^{(d)}I_d$. 
Making for example the substitution $t_1\mapsto \rho^{(d)}(h)$, 
$t_2\mapsto \rho^{(d)}(e)$, $t_3\mapsto \rho^{(d)}(f)$ in $w_{ijk}^{(d)}$ and in $\mathrm{St}_3(t_1,t_2,t_3)$ 
one can easily compute the scalar, and get the desired equality.  
\end{proof}

\end{document}